\documentclass[11pt,a4paper]{article}
\RequirePackage[OT1]{fontenc}
\usepackage[english]{babel}
\usepackage{amsmath}
\usepackage{amssymb}
\usepackage{bbm}
\usepackage{bm}
\usepackage{mathrsfs}
\usepackage{dsfont}
\usepackage{graphicx,theorem}
\usepackage{color}

\allowdisplaybreaks  
\newcommand{\E}{\mathds{E}}

\newcommand{\la}{\lambda}

\newcommand{\ga}{\gamma}

\newcommand{\ra}{\rangle}

      \def\nn{\nonumber}
      \def\rf#1{\mbox{$(\ref{#1})$}}
      \def\a{\alpha}

      \def\be{\begin{equation}} 
      \def\ee{\end{equation}} 
      \def\beqn{\begin{eqnarray}} 
      \def\eeqn{\end{eqnarray}} 
      \def\beq{\begin{eqnarray*}} 
      \def\eeq{\end{eqnarray*}}
      \def\proof{{\bf Proof.}\ }
      \def\ga{\gamma} 
      \def\la{\lambda} 
      \def\ra{\rightarrow} 

\theoremstyle{plain}
\newtheorem{thm}{\textbf{Theorem}}[section]

\begin{document}

\begin{center}
\textmd{\Large{\bfseries{{Moderate deviations for Ewens-Pitman exchangeable random partitions}}}}
\end{center}
\begin{center}
{Stefano Favaro$^1$, Shui Feng$^2$ and Fuqing Gao$^3$}
\end{center}

\begin{quote}
\begin{small}
\begin{center}
\noindent $^1$ University of Torino and Collegio Carlo Alberto, Torino, Italy.\\ \textit{E-mail}: stefano.favaro@unito.it

\noindent $^2$ McMaster University, Hamilton, Canada.\\ \textit{E-mail}: shuifeng@univmail.cis.mcmaster.ca

\noindent $^3$ Wuhan University, Hubei, China\\ \textit{E-mail}: fqgao@whu.edu.cn

\end{center}
\end{small}
\end{quote}

\smallskip

\centerline{\textit{October 2016}}

\begin{abstract}

Consider a population of individuals belonging to an infinity number of types, and assume that type proportions follow the two-parameter Poisson-Dirichlet distribution. A sample of size $n$ is selected from the population. The total number of different types and the number of types appearing in the sample with a fixed frequency are important statistics. In this paper we establish the moderate deviation principles for these quantities. The corresponding rate functions are explicitly identified, which help revealing a critical scale and understanding the exact role of the parameters. Conditional, or posterior, counterparts of moderate deviation principles are also established.
\par

\vspace{9pt}
\noindent {\it Key words and phrases:}
$\alpha$-diversity; exchangeable random partition; Dirichlet process; large and moderate deviation; random probability measure; two parameter Poisson-Dirichlet distribution 
\par
\end{abstract}


\section{Introduction}
Consider a population of countable number of individuals belonging to an infinite number of types. The type of each individual is labelled by a point in a Polish space $S$. The type proportions in the population are thus a point ${\bf p }=(p_1,p_2, \ldots)$ in the space $\triangle:=\{{\bf q}=(q_1,q_2,\ldots): q_i\geq 0, \sum_{j=1}^{\infty}q_j=1\}$. For each $n\geq 1$, let $X_1,X_2,\ldots,X_n$ be a random sample of size $n$ from the population with $X_i$ denoting the type of the $i$th sample.  The sample diversity is defined as 
\[
K_n :=\mbox{total number of different types in the sample.}
\]
For any $1\leq l\leq n$, set
\[
M_{l,n}:=\mbox{total number of types that appear in the sample $l$ times.}
\]
The quantity  $M_{l,n}$ is typically referred to as the sample diversity with frequency $l$. Both the random variables $K_{n}$ and $M_{l,n}$, as well as related functions, provide important statistics for inference about the population diversity. 

A natural scheme arises in the occupancy problem. Consider a countable numbers of urns. Balls are put into the urns independently and each ball lands in urn $i$ with probability $p_i$. After $n$ balls are put into the urns, the total number of occupied urns is $K_n$, and $M_{l,n}$ is the numbers of urns with $l$ balls inside.  Assuming that $p_1\geq p_2\geq \ldots$, a comprehensive study of $K_n$ and $M_{l,n}$ was  carried out in \cite{Kar67}.  See also \cite{Jan(08)}, \cite{Bar(09)}, \cite{Bou(16)} for some recent contributions. A comprehensive survey of recents progresses in this context is found in \cite{GHP07}.  

Adding randomness to the type proportions ${\bf p}$, the population will have random type proportions  with the law ${\cal P}$ being a probability on $\triangle$. Note that, instead of being independent and identically distributed (iid), the random sample $X_1, X_2, \ldots,X_n$ becomes exchangeable. In particular, following the de Finetti  theorem, the random type proportions are recovered from the masses of the limit of empirical distributions of the random sample as $n$ tends to infinity. This framework fits naturally in the context of Bayesian nonparametric inference. See, e.g., \cite{FLMP09}. In particular the law ${\cal P}$ can be viewed as the prior distribution on the unknown species composition $(p_{i})_{i\geq1}$ of the population.  The main  interests in Bayesian nonparametrics are the posterior distribution of ${\cal P}$ given an initial sample $(X_{1},\ldots,X_{n})$  and associated statistical inferences. More specifically, given an initial sample $(X_{1},\ldots,X_{n})$, interest lies in making inference based on certain statistics  induced by an additional unobserved sample of size $m$. These include, among others, the sample diversity $K_{m}^{(n)}$ and the sample diversity $M_{l,m}^{(n)}$ with frequency $l$ to be observed in the additional sample of size $m$. We call  $K_{m}^{(n)}$ and  $M_{l,m}^{(n)}$ the posterior sample diversity and the posterior sample diversity with frequency $l$, respectively.

The most studied family of probabilities on $\triangle$ is Kingman's Poisson-Dirichlet distribution (\cite{Kingman75}) describing in the genetics context the distribution of allele frequencies in a neutral population. This is followed by the study of the two-parameter Poisson-Dirichlet distribution (\cite{Pit(97)}).  Various generalizations of these models can be found in \cite{Ber(06)}, \cite{Pit(06)} and the references therein.  

The focus of this paper is on the asymptotic behaviour of all these sample diversities when the random proportions in the population follow Kingman's Poisson-Dirichlet distribution and its two-parameter generalization. Specifically, for any $\alpha$ in $[0,1)$ and $\theta >-\alpha$, let $U_k$, $k=1,2,\cdots$, be a sequence of independent random variables such that $U_k$ has $Beta(1-\alpha,\theta+ k\alpha)$ distribution. If
\[ \label{GEM1}
V_1(\alpha,\theta) = U_1,\  V_n(\alpha,\theta) = (1-U_1)\cdots (1-U_{n-1})U_n,\  n \geq 2.
\]
then 
\begin{displaymath}
{\bf V}(\alpha,\theta)=(V_1(\alpha,\theta), V_2(\alpha,\theta),\cdots)\in \triangle
\end{displaymath}
with probability $1$. The law of the descending order statistic ${\bf P}(\alpha,\theta)=(P_1(\alpha,\theta), P_2(\alpha,\theta),\cdots)$  of ${\bf V}(\alpha,\theta)$ is the so-called the two-parameter Poisson-Dirichlet distribution and is denoted by $PD(\alpha,\theta)$.  Kingman's Poisson-Dirichlet distribution which corresponds to $\alpha=0$.  The sample diversities $K_n, K_m^{(n)}$, $M_{l,n}$ and $M_{l,m}^{(n)}$ depend on the parameters $\theta$ and $\alpha$. For notational convenience we will not indicate the dependence explicitly.  When $\alpha=0$, the parameter $\theta$ corresponds to the scaled population mutation rate. The sample diversity $K_n$ turns out to be a sufficient statistic for the estimation of $\theta$.

There have been many studies on the behaviour of $K_n$ and $M_{l,n}$, as $n$ goes to infinity, and of  $K_{m}^{(n)}$ and $M_{l,m}^{(n)}$, as $m$ goes to infinity. In the case $\alpha=0$, one can represent $K_n$  as the summation of independent Bernoulli random variables and show that $\frac{K_n}{\ln n}$ converges to $\theta$ almost surely.  In \cite{Gon44} ($\alpha=0,\theta=1$)  and \cite{Han90}($\alpha=0$, general $\theta$)  the following central limit theorem was obtained
\[ \label{int1}
\frac{K_n -\theta \ln n}{\sqrt{\ln n}} \Rightarrow\, N(0,1),
\] 
as $n$ goes to infinity, with $\Rightarrow$ denoting the weak convergence. When the parameter $\alpha$ is positive, the Gaussian limit no longer holds. In particular, it was shown in \cite{Pitman1992} that one has

\[ \label{int2}
\lim_{n \ra \infty}\frac{K_n}{n^{\a}}= S_{\alpha,\theta}, \ \ \ \  a.s.
\]
where $S_{\alpha,\theta}$ is related to the Mittag-Leffler distribution.  For any $l\geq 1$, the following holds (\cite{Pit(06)}):

\[
\lim_{n \ra \infty}\frac{M_{l,n}}{n^{\a}}= (-1)^{l-1}{\alpha \choose  l}S_{\alpha,\theta}, \ \ \ \  a.s.
\]
The random variable $S_{\alpha,\theta}$ is referred to as the $\alpha$-diversity of the $PD(\alpha,\theta)$ distribution. Large deviation principles for $K_n$ were established in \cite{FenHop(98)}. The fluctuation behaviour of $K_{m}^{(n)}$ and $M_{l,m}^{(n)}$, as $m$ goes to infinity, were studied in \cite{Fav(13)}, where the notion of posterior $\alpha$-diversity were introduced. Moreover, the associated large deviation principles have been recently established in \cite{FaFe14} and \cite{FaFe15}.

The main results of the present paper are the moderate deviation principles (henceforth MDPs) for the sample diversities $K_n$, $K_m^{(n)}$, $M_{l,n}$ and $M_{l,m}^{(n)}$ under $PD(\alpha,\theta)$ with $\alpha>0$.   Our study is motivated by a better understanding of the non-Gaussian moderate deviation behaviour and a refined analysis about the role of the parameters $\alpha$ and $\theta$ involved. Interestingly, our results identify a critical scale and reveal the role of the parameters $\theta$ and $\alpha$ explicitly. The paper is organized as follows.  Section 2 contains the study of MDPs for the sample diversities $K_n$ and $M_{l,n}$ . The corresponding results for the posterior sample diversities are then presented in Section 3. A key step here is a Bernoulli representation of $K_m^{(n)}$ and $M_{l,m}^{(n)}$. All terminologies and theorems on large and moderate deviations are based on  the reference \cite{DZ98}.


\section{Moderate deviations for $K_n$ and $M_{l,n}$ }

In the case $\alpha=0$ and $\theta>0$, $K_n$ is the summation of independent Bernoulli random variables, and for each $1\leq l\leq n$ $M_{l,n}$ is approximately a Poisson random variable. Accordingly, the corresponding moderate deviations are standard. Hence we assume in the sequel  that $0<\alpha<1$ and $ \theta+\alpha>0$.

Moderate deviations in these cases lie between the fluctuation limit results for $\frac{K_n}{n^{\alpha}}$ and $\frac{M_{l,n}}{n^{\alpha}}$, and the large deviation results for $\frac{K_n}{n}$ and $\frac{M_{l,n}}{n}$, respectively. In particular our objectives consist of establishing large deviation principles for $\frac{K_n}{n^{\alpha}\beta_n}$ and  $\frac{M_{l,n}}{n^{\alpha}\beta_n}$ where $\beta_n$ converges to infinity at a slower pace than $n^{1-\alpha}$ as $n$ tends to infinity. More specifically, we assume that $\beta_n$ satisfies
 
\be\label{mdp-assum}
\lim_{n \ra \infty}\frac{\beta_n}{n^{1-\alpha}}= 0,\quad  \lim_{n \ra \infty}\frac{\beta_n}{(\ln n)^{1-\alpha}}= \infty.
\ee
The assumption that $\beta_n$ grows faster that $(\ln n)^{1-\alpha}$ is crucial  for establishing the following MDP.

\begin{thm}\label{t1}
For any $\alpha\in (0,1)$ and for any $\theta>-\alpha$,  $\frac{ K_n}{n^\alpha \beta_n}$ satisfies a large deviation principle on $\mathbb{R}$ with speed $\beta_n^{1/(1-\alpha)} $ and rate function $I_\alpha(\cdot)$ defined by
 $$
 I_\alpha(x)=\left\{\begin{array}{ll} (1-\alpha)\alpha^{\alpha/(1-\alpha)} x^{1/(1-\alpha)} & \mbox{ if } x>0,\\[0.6cm]
 +\infty & \mbox{ if } x\leq 0.
 \end{array}
 \right.
 $$
\end{thm}

 \begin{proof} Let us define $\tilde{K}_n=\frac{K_n}{n^{\alpha}\beta_n}$.  First, by a direct calculation, one has that for any  $\lambda\leq 0$
 $$
 \lim_{n\to\infty} \frac{1}{\beta_n^{1/(1-\alpha)}}\ln \mathbb{E}\left[\exp\{\lambda \beta_n^{1/(1-\alpha)} \tilde{K}_n\}\right]=0.
 $$
 
 For any $\lambda>0$, set $y_n=1-\exp\{-\lambda n^{-\alpha}\beta^{\alpha/(1-\alpha)}_n\}$. First assume $\theta=0$. Then by equation (3.5) in \cite{FenHop(98)}, we have 
 
 \beq
 \mathbb{E}\left[\exp\{\lambda \beta_n^{1/(1-\alpha)} \tilde{K}_n\}\right]&=&  \mathbb{E}\left[ (1-y_n)^{-K_n}\right]\\
 &=& \sum_{i=0}^{\infty} y_n^i{i\alpha +n-1\choose n-1}.
 \eeq
Let $\lfloor i\alpha \rfloor$ denote the integer part of $i\alpha$. It follows from direct calculation that
 
 \beq
  &&\sum_{i=0}^{\infty} y_n^i{i\alpha +n-1\choose n-1}\\
  &&\geq   \sum_{i=0}^{\infty} y_n^i{\lfloor i\alpha\rfloor +n-1\choose n-1} =\sum_{k=0}^{\infty} {k+n-1\choose n-1} \sum_{\lfloor i\alpha \rfloor =k}y_n^i\\
&& \geq  y^{1/\alpha}_n \sum_{k=0}^{\infty} {k+n-1\choose n-1} (y^{1/\alpha}_n)^k= \frac{y_n^{1/\alpha}}{(1-y_n^{1/\alpha})^n}.\eeq
  On the other hand, 
   \beq
  &&\sum_{i=0}^{\infty} y_n^i{i\alpha +n-1\choose n-1}\\
  &&\leq   \sum_{i=0}^{\infty} y_n^i{\lfloor i\alpha\rfloor +n\choose n-1}
  = \sum_{i=0}^{\infty} y_n^i \frac{\lfloor i\alpha\rfloor +n}{\lfloor i\alpha\rfloor +1}{\lfloor i\alpha\rfloor +n-1\choose n-1}\\
    &&\leq n\sum_{k=0}^{\infty} {k+n-1\choose n-1} \sum_{\lfloor i\alpha \rfloor =k}(y^{1/\alpha}_n)^{i\alpha}\leq \frac{n}{\alpha} \sum_{k=0}^{\infty} {k+n-1\choose n-1} (y^{1/\alpha}_n)^k\\
  &&=\frac{n}{\alpha}\frac{1}{(1-y_n^{1/\alpha})^n}.\eeq  
Putting these together and applying assumption \rf{mdp-assum} one gets 
 \beq
&& \lim_{n\to\infty} \frac{1}{\beta_n^{1/(1-\alpha)}}\ln \mathbb{E}\left[\exp\{\lambda n^{-\alpha}\beta_n^{\alpha/(1-\alpha)}K_n\}\right]\\
&&=  \lim_{n\to\infty} \ln \bigg[ 1-\left(1-\exp\{-\lambda n^{-\alpha}\beta_n^{\alpha/(1-\alpha)}\}\right)^{1/\alpha}\bigg]^{-n\beta_n^{-1/(1-\alpha)}}\\
&&= \lambda^{1/\alpha}.
 \eeq
  
Since the law of $K_n$ under $PD(\alpha,\theta)$ is equivalent to the law of $K_n$ under $PD(\alpha,0)$, the above limit holds for $\lambda\geq 0$,
 
 Set
 \[
 \Lambda(\lambda)=\left\{\begin{array}{ll} 
 \lambda^{1/\alpha} & \mbox{ if } \lambda>0,\\[0.6cm]
 0 & \mbox{ otherwise. } 
 \end{array}
 \right.
 \]
 Noting that $I_\alpha(x)=\sup_{\lambda \in\mathbb R}\{\lambda x-\Lambda(\lambda)\}$, the conclusion holds following G\"artner-Ellis theorem (\cite{DZ98}).

\hfill $\Box$
 \end{proof}
 
Theorem \ref{t1} introduces a moderate deviation principle for $K_{n}$. Rewrite the rate function as
 \[
 I_{\alpha}(x)= \exp\{\frac{1}{1-\alpha}[H_\alpha+\ln x]\}
 \]
 with $H_\alpha=(1-\alpha)\ln(1-\alpha) +\alpha\ln\alpha$ being the entropy function, it follows that
$\alpha x =1$ is a critical curve. For $0<x \leq 1$, $I_{\alpha}(x)$ is decreasing in $\alpha$. For $x>1$ $I_\alpha(x)$ decreases for $\alpha$ in $(0,1/x)$, increases for $\alpha$ in $(1/x,1)$. The minimum is achieved at the point $1/x$. 
 Discounting the scale differences, these results provide a refined comparison between different models in terms of 
 deviation manners.
 
  \vspace{0.2cm}
 
In the next theorem we establish the MDP for $M_{l,n}$ for any $l\geq 1$.

\begin{thm}\label{t2}
For any $\alpha\in (0,1)$ and for any $\theta>-\alpha$, $\frac{M_{l,n}}{n^\alpha \beta_n}$ satisfies a large deviation principle on $\mathbb{R}$ with speed $\beta_n^{1/(1-\alpha)} $ and rate function $I_{\alpha, l}(\cdot)$ defined by
 $$
 I_{\alpha, l}(x)=\left\{\begin{array}{ll} (1-\alpha)\bigg(\frac{l!}{(1-\alpha)_{(l-1)\uparrow 1}}\bigg)^{\alpha/(1-\alpha)} x^{1/(1-\alpha)} & \mbox{ if } x>0,\\[0.6cm]
 +\infty & \mbox{ if } x\leq 0,
 \end{array}
 \right.
 $$
 where $(a)_{j\uparrow b}=a(a+b)\cdots(a+(j-1)b)$ with the proviso $(a)_{0\uparrow b}=1$.\end{thm}
\proof Let $y_n$ be as in Theorem~\ref{t1}. Set 
$$y_{n,l}=\frac{\alpha (1-\alpha)_{(l-1)\uparrow 1}}{l!}\frac{y_n}{1-y_n}.
$$

By an argument similar to the proof of Lemma~2.1 in \cite{FaFe14}, we obtain that for any $\la>0$ 
\beq 
\mathbb{E}\bigg[ \exp\{\la n^{-\alpha}\beta_n^{\alpha/(1-\alpha)}M_{l,n}\}\bigg]&=&\mathbb{E}\bigg[\bigg(\frac{1}{1-y_n}\bigg)^{M_{l,n}}\bigg]\\
&=&\sum_{i=0}^{\lfloor{n/l\rfloor}}y_{n,l}^{i} \frac{n}{n-il+\alpha i}{n-il+i\alpha \choose n-il}.
\eeq
Note that, since $1\leq \frac{n}{n-il+\alpha i}\leq \frac{l}{\alpha}$ for $i=0,\ldots, \lfloor{n/l\rfloor}$, it follows that the large $n$ approximation of 
$$\mathbb{E}\bigg[ \exp\{\la n^{-\alpha}\beta_n^{\alpha/(1-\alpha)}M_{l,n}\}\bigg]$$ 
is equivalent to that of 

\[
H_{n,l}=\sum_{i=0}^{\lfloor{n/l\rfloor}}y_{n,l}^{i} {n-il+i\alpha \choose n-il}.
\]
Set
\[
H^-_{n,l}=\sum_{i=0}^{\lfloor{n/l\rfloor}}y_{n,l}^{i} {n-il+\lfloor i\alpha\rfloor \choose n-il}\]
and 
\[
H^+_{n,l}=\sum_{i=0}^{\lfloor{n/l\rfloor}}y_{n,l}^{i} {n-il+\lfloor i\alpha\rfloor +1 \choose n-il}.
\]
It is clear that
\[
H^-_{n,l}\leq H_{n,l}\leq H^+_{n,l}\leq (n+1) H^-_{n,l}.
\]
The assumption for $\beta_n$ guarantees that the factor $n+1$ in the upper bound does not contribute to the scaled logarithmic limit. Accordingly, we can write 
\be\label{mdp1}
\lim_{n\ra \infty}\frac{1}{\beta_n^{1/(1-\alpha)}}\ln \mathbb{E}\bigg[ \exp\{\la n^{-\alpha}\beta_n^{\alpha/(1-\alpha)}M_{l,n}\}\bigg]=\lim_{n\ra \infty}\frac{1}{\beta_n^{1/(1-\alpha)}}\ln H^-_{n,l}.
\ee
To estimate $H^-_{n,l}$, we write
\beq
H^-_{n,l}&=& \sum_{i=0}^{\lfloor n/l\rfloor}(y^{1/\alpha}_{n,l})^{ i\alpha} \frac{(n-il +1)\cdots (n-il+\lfloor i\alpha\rfloor)}{(\lfloor i\alpha\rfloor)!}\\
&=& \sum_{i=0}^{\lfloor n/l\rfloor}(y^{1/\alpha}_{n,l})^{ i\alpha-\lfloor i\alpha \rfloor} (ny^{1/\alpha}_{n,l})^{\lfloor i\alpha\rfloor} \frac{(1+(1-il)/n)\cdots (1+(\lfloor i\alpha\rfloor-il)/n )}{(\lfloor i\alpha\rfloor)!}\eeq
which is controlled from below by 
\[
\sum_{i=0}^{\lfloor n/l\rfloor}(y^{1/\alpha}_{n,l})^{ i\alpha-\lfloor i\alpha \rfloor} (ny^{1/\alpha}_{n,l})^{\lfloor i\alpha\rfloor} \frac{(1+(1-il)/n)^{\lfloor i\alpha \rfloor}}{(\lfloor i\alpha\rfloor)!}\]
and from above by
\[
\sum_{i=0}^{\lfloor n/l\rfloor}(y^{1/\alpha}_{n,l})^{ i\alpha-\lfloor i\alpha \rfloor} (ny^{1/\alpha}_{n,l})^{\lfloor i\alpha\rfloor} \frac{(1+(\lfloor i\alpha\rfloor-il)/n )^{\lfloor i\alpha \rfloor}}{(\lfloor i\alpha\rfloor)!}.\]
Since $(y^{1/\alpha}_{n,l})^{ i\alpha-\lfloor i\alpha \rfloor} $ does not affect the scaled logarithmic limit in \rf{mdp1}, it suffices to focus on
\[
D_{n,l}=\sum_{i=0}^{\lfloor n/l\rfloor} (ny^{1/\alpha}_{n,l})^{\lfloor i\alpha\rfloor} \frac{(1+(1-il)/n)^{\lfloor i\alpha \rfloor}}{(\lfloor i\alpha\rfloor)!}\]
and 
\[
J_{n,l}=\sum_{i=0}^{\lfloor n/l\rfloor}(ny^{1/\alpha}_{n,l})^{\lfloor i\alpha\rfloor} \frac{(1+(\lfloor i\alpha\rfloor-il)/n )^{\lfloor i\alpha \rfloor}}{(\lfloor i\alpha\rfloor)!}
\]

Set $\gamma_n =\lfloor \beta_n^{1/(1-\alpha)}\rfloor$ and write
\[
D_{n,l}= D^1_{n,l} +D^2_{n,l}
\]
with
\[
D^1_{n,l}=\sum_{i=0}^{\gamma_n} (ny^{1/\alpha}_{n,l})^{\lfloor i\alpha\rfloor} \frac{(1+(1-il)/n)^{\lfloor i\alpha \rfloor}}{(\lfloor i\alpha\rfloor)!}.
\]
It follows that
\beqn
D^2_{n,l}&=&\sum_{i=\gamma_n+1}^{\lfloor n/l\rfloor} (ny^{1/\alpha}_{n,l})^{\lfloor i\alpha \rfloor} \frac{(1+(1-il)/n)^{\lfloor i\alpha \rfloor}}{(\lfloor i\alpha\rfloor)!}\nn\\
&\leq& \sum_{i=\gamma_n+1}^{\lfloor n/l\rfloor} \frac{(ny^{1/\alpha}_{n,l})^{\lfloor i\alpha \rfloor}}{(\lfloor i\alpha\rfloor)!}
\leq \frac{1}{\alpha}\sum_{k=\lfloor (\gamma_n+1)\alpha\rfloor}^{\infty} \frac{(ny^{1/\alpha}_{n,l})^{k}}{k!}\label{tail1}\\
&\leq & \frac{1}{\alpha}\frac{(ny_{n,l}^{1/\alpha})^{\lfloor (\gamma_n+1)\alpha\rfloor}}{\lfloor (\gamma_n+1)\alpha\rfloor!}\exp\{ny_{n,l}^{1/\alpha}\}.\nn
\eeqn
By direct calculation, we have 
\be\label{tail 2}
\lim_{n\ra \infty}\frac{ny_{n,l}^{1/\alpha}}{\beta_n^{1/(1-\alpha)}}=\bigg(\frac{\alpha (1-\alpha)_{(l-1)\uparrow 1}}{l!}\la\bigg)^{1/\alpha}
\ee
and 
\be\label{tail3}
\lim_{n \ra \infty}\frac{1}{\beta_n^{1/(1-\alpha)}}\ln \lfloor (\gamma_n+1)\alpha\rfloor!=\infty.\ee
Hence
\[
\lim_{n \ra \infty}\frac{1}{\beta_n^{1/(1-\alpha)}}\ln D^2_{n,l}=-\infty.
\]
This implies that
\[
\lim_{n \ra \infty}\frac{1}{\beta_n^{1/(1-\alpha)}}\ln D_{n,l}=\lim_{n \ra \infty}\frac{1}{\beta_n^{1/(1-\alpha)}}\ln D^1_{n,l}.\]

Noting that $\lim_{n\ra \infty}\max_{10\leq i
\leq \gamma_n}\{|(1-il)/n|\}=0$, we obtain

\[
\lim_{n \ra \infty}\frac{1}{\beta_n^{1/(1-\alpha)}}\ln D^1_{n,l}=\lim_{n \ra \infty}\frac{1}{\beta_n^{1/(1-\alpha)}}\ln \sum_{i=0}^{\gamma_n} \frac{(ny^{1/\alpha}_{n,l})^{\lfloor i\alpha\rfloor}}{(\lfloor i\alpha\rfloor)!}.\]

By an argument similar to that used in deriving the estimation \rf{tail1}, and taking into account of \rf{tail 2}, we obtain that
\begin{align}\label{tail4}
&\lim_{n \ra \infty}\frac{1}{\beta_n^{1/(1-\alpha)}}\ln D_{n,l}\\
&\notag\quad=\lim_{n \ra \infty}\frac{1}{\beta_n^{1/(1-\alpha)}}\ln \sum_{i=0}^{\gamma_n} \frac{(ny^{1/\alpha}_{n,l})^{\lfloor i\alpha\rfloor}}{(\lfloor i\alpha\rfloor)!}\\
&\notag\quad =\lim_{n \ra \infty}\frac{1}{\beta_n^{1/(1-\alpha)}}\ln\exp\{ny_{n,l}^{1/\alpha}\}\\
&\notag\quad= \bigg(\frac{\alpha (1-\alpha)_{(l-1)\uparrow 1}}{l!}\la\bigg)^{1/\alpha},\nn
\end{align}
Similarly we can prove that
\be\label{tail5}
\lim_{n \ra \infty}\frac{1}{\beta_n^{1/(1-\alpha)}}\ln J_{n,l}= \bigg(\frac{\alpha (1-\alpha)_{(l-1)\uparrow 1}}{l!}\la\bigg)^{1/\alpha}.\ee
The result now follows from \rf{mdp1}, \rf{tail4}, \rf{tail5} and G\"artner-Ellis theorem. 

\hfill  $\Box$


\section{Moderate deviations for $K_m^{(n)}$ and $M_{l,m}^{(n)}$ }

Given $n \geq 1$, let $\mathbf{X}_{n}=(X_{1},\ldots,X_{n})$ be a sample from the population with type proportions following  two parameter Poisson-Dirichlet distribution $PD(\alpha,\theta)$. Let the sample $\mathbf{X}_{n}$ featuring $K_{n}=j\leq n$ distinct types with corresponding frequencies $\mathbf{N}_{n}=(N_{1,1},\ldots,N_{1,K_{n}})=(n_{1},\ldots,n_{j})$, and let $M_{l,n}$ be the number of distinct types with frequency $1\leq l \leq n$. Now 
consider an additional sample $\mathbf{X}^{(n)}_{m}=(X_{n+1},\ldots,X_{n+m})$ of size $m$, and let $K_m^{(n)}$
 and $M_{l,m}^{(n)}$ be the sample diversity and sample diversity with frequency $1\leq l \leq m$
 in $\mathbf{X}^{(n)}_{m}$.  In this section we derive the MDPs for  $K_m^{(n)}$
 and $M_{l,m}^{(n)}$ as $m$ tends to infinity given $\mathbf{X}_{n}$, $K_n$ and  $\mathbf{N}_{n}$. The law of the type proportions of the population is now the posterior distribution of $PD(\alpha,\theta)$ given $\mathbf{X}_{n}$. Structurally we can divide the type into two groups: types appeared in the sample  $\mathbf{X}_{n}$ and brand new types.
   
  Let $L_{m}^{(n)}$ be the number of $X_{n+i}$'s, for $i=1,\ldots,m$, that do not coincide with $X_{i}$'s, for $i=1,\ldots,n$. Also, let
\begin{itemize}
\item[i)] $\tilde{K}_{m}^{(n)}$ be the number of new distinct types in the additional sample $\mathbf{X}_{m}$, i.e. the number of types in $\mathbf{X}^{(n)}_{m}$ which do not coincide with any of the types that appear in the initial sample $\mathbf{X}_{n}$;
\item[ii)] $\tilde{M}_{l,m}^{(n)}$ be the number of new distinct types with frequency $l$ in the additional sample $\mathbf{X}_{m}$, i.e., the number of types with frequency $l$ among the new types that appear in $\mathbf{X}^{(n)}_{m}$, such that
\begin{displaymath}
\sum_{l=1}^{m}\tilde{M}_{l,m}^{(n)}=\tilde{K}_{m}^{(n)}\quad\mbox{ and }\quad\sum_{l=1}^{n}l\tilde{M}_{l,m}^{(n)}=L_{m}^{(n)}.
\end{displaymath}
\end{itemize}

Since the sample $\mathbf{X}_{n}$ is fixed, the moderate deviations for $K_m^{(n)}$ and $M_{l,m}^{(n)}$ are equivalent to the corresponding moderate deviations for $\tilde{K}_m^{(n)}$ and $\tilde{M}_{m,l}^{(n)}$. Thus we will focus on  $\tilde{K}_m^{(n)}$ and $\tilde{M}_{m,l}^{(n)}$ in the sequel. The key step in the proof is the following representation for the conditional, or posterior, distributions of $\tilde{K}_{m}^{(n)}$ given $(K_{n},\mathbf{N}_{n})$ and of $\tilde{M}_{l,m}^{(n)}$ given $(K_{n},\mathbf{N}_{n})$, for any $l=1,\ldots,m$. With a slight abuse of notation, throughout this section we write $X\,|\,Y$ to denote a random variable whose distribution coincides with the conditional distribution of $X$ given $Y$.

\begin{thm}\label{p-mdp1}
For any $k\geq1$ and $p\in[0,1]$, let $Z_{k,p}$ be Binomial random variable with parameter $(k,p)$, and for any $a, b>0$ let $B_{a,b}$ be a Beta random variable with parameter $(a,b)$. If $K^{\ast}_{m}$ and $M^{\ast}_{l,m}$ denote the number of distinct types and the number of distinct types with frequency $1\leq l\leq m$, respectively, in a sample of size $m$ from $PD(\alpha,\theta+n)$, then we have 
\begin{equation}\label{eq_id1}
\tilde{K}_{m}^{(n)}\,|\,(K_{n}=j,\mathbf{N}_{n}=(n_{1},\ldots,n_{j}))\stackrel{\text{d}}{=}\tilde{K}_{m}^{(n)}\,|\,(K_{n}=j)\stackrel{\text{d}}{=}Z_{K^{\ast}_{m},B_{\frac{\theta}{\alpha}+j,\frac{n}{\alpha}-j}}
\end{equation}
and
\begin{equation}\label{eq_id2}
\tilde{M}_{l,m}^{(n)}\,|\,(K_{n}=j,\mathbf{N}_{n}=(n_{1},\ldots,n_{j}))\stackrel{\text{d}}{=}\tilde{M}_{l,m}^{(n)}\,|\,(K_{n}=j)\stackrel{\text{d}}{=}Z_{M^{\ast}_{l,m},B_{\frac{\theta}{\alpha}+j,\frac{n}{\alpha}-j}}
\end{equation}
where $\stackrel{\mbox{d}}{=}$ denotes the equality in distribution, and $B_{\frac{\theta}{\alpha}+j,\frac{n}{\alpha}-j}$ is independent of $K^{\ast}_{m}$ and of $M^{\ast}_{l,m}$.
 \end{thm}
\begin{proof} Since all random variables involved are bounded, it suffices to verify the equality of all moments. We start by recalling some moment formulate for $K^{\ast}_{m}$ and $M^{\ast}_{l,m}$ (cf. \cite{Yam(00)} and \cite{Fav(13)}). In particular one has
\begin{equation}\label{eq_momk}
\E[(K^{\ast}_{m})_{r\downarrow 1}]=\left(\frac{\theta+n}{\alpha}\right)_{r\uparrow1}\sum_{i=0}^{r}(-1)^{r-i}{r\choose i}\frac{(\theta+n+i\alpha)_{m\uparrow1}}{(\theta+n)_{m\uparrow1}}
\end{equation}
and
\begin{align}\label{eq_momm}
&\E[(M^{\ast}_{l,m})_{r\downarrow 1}]\\
&\notag\quad=(m)_{rl\downarrow 1}\left(\frac{\alpha(1-\alpha)_{(l-1)\uparrow1}}{l!}\right)^{r}\left(\frac{\theta+n}{\alpha}\right)_{r\uparrow1}\frac{(\theta+n+r\alpha)_{(m-rl)\uparrow1}}{(\theta+n)_{m\uparrow1}},
\end{align}
where $(c)_{j\downarrow 1}=(c)_{j\uparrow -1}$
Moreover, let us recall the factorial moment of order $r$ of the Binomial random variable $Z_{n,p}$, i.e., 
\begin{equation}\label{eq_momb}
\E[(Z_{n,p})^{r}]=\sum_{t=0}^{r}S(r,t)(n)_{t\downarrow 1}p^{t},
\end{equation}
with $S(n,k)$ being the Stirling number of the second kind. If $S(n,k;a)$ denotes the non-central Stirling number of the second kind, see \cite{Cha(05)}, then by means of Proposition 1 in \cite{FLMP09} we have
\beq
&&\E[(\tilde{K}_{m}^{(n)})^{r}\,|\,K_{n}=j]\\
&&\quad=\sum_{i=0}^{r}(-1)^{r-i}\left(j+\frac{\theta}{\alpha}\right)_{i\uparrow1}S\left(r,i;j+\frac{\theta}{\alpha}\right)\frac{(\theta+n+i\alpha)_{m\uparrow1}}{(\theta+n)_{m\uparrow1}}\\
&&\mbox{(by expanding $S(r,i;j+\theta/\alpha)$ as a finite sum)}\\
&&\quad=\sum_{i=0}^{r}(-1)^{-i}\frac{(\theta+n+i\alpha)_{m\uparrow1}}{(\theta+n)_{m\uparrow1}}\sum_{t=i}^{r}(-1)^{t}{t\choose i}S(r,t)\left(j+\frac{\theta}{\alpha}\right)_{t\uparrow1}\\
&&\quad=\sum_{t=0}^{r}S(r,t)\frac{\left(j+\frac{\theta}{\alpha}\right)_{t\uparrow1}}{\left(\frac{\theta+n}{\alpha}\right)_{t\uparrow1}}\left(\frac{\theta+n}{\alpha}\right)_{t\uparrow1}\sum_{i=0}^{t}(-1)^{t-i}{t\choose i}\frac{(\theta+n+i\alpha)_{m\uparrow1}}{(\theta+n)_{m\uparrow1}}\\
&&\mbox{(by Equation \rf{eq_momk})}\\
&&\quad=\sum_{t=0}^{r}S(r,t)\frac{\left(j+\frac{\theta}{\alpha}\right)_{t\uparrow1}}{\left(\frac{\theta+n}{\alpha}\right)_{t\uparrow1}}\E[(K^{\ast}_{m})_{t\downarrow 1}]\\
&&\mbox{(by expanding $(j+\theta/\alpha)_{t\uparrow1}/((\theta+n)/\alpha)_{t\uparrow1}$ as an Euler integral)}\\
&&\quad=\sum_{t=0}^{r}S(r,t)\E[(K^{\ast}_{m})_{t\downarrow 1}]\frac{\Gamma\left(\frac{\theta+n}{\alpha}\right)}{\Gamma\left(\frac{\theta}{\alpha}+j\right)\Gamma\left(\frac{n}{\alpha}-j\right)}\int_{0}^{1}x^{t+\frac{\theta}{\alpha}+j-1}(1-x)^{\frac{n}{\alpha}-j-1}d\, x\\
&&\quad=\sum_{t=0}^{r}S(r,t)\E[(K^{\ast}_{m})_{t\downarrow 1}]\E[(B_{\frac{\theta}{\alpha}+j,\frac{n}{\alpha}-j})^{t}]\\
&&\quad=\E\left[\E\left[\sum_{t=0}^{r}S(r,t)(K^{\ast}_{m})_{t\downarrow 1}(B_{\frac{\theta}{\alpha}+j,\frac{n}{\alpha}-j})^{t}\right]\right]\\
&&\mbox{(by Equation \rf{eq_momb})}\\
&&\quad=\E\left[\left(Z_{K^{\ast}_{m},B_{\frac{\theta}{\alpha}+j,\frac{n}{\alpha}-j}}\right)^{r}\right]
\eeq
and the proof of the representation \rf{eq_id1} is completed. Similarly, by Theorem 2 in \cite{Fav(13)} we can write
\beq
&&\E[(\tilde{M}_{l,m}^{(n)})^{r}\,|\,K_{n}=j]\\
&&\quad=\sum_{t=0}^{r}S(r,t)(m)_{tl\downarrow 1}\left(\frac{\alpha(1-\alpha)_{(l-1)\uparrow1}}{l!}\right)^{t}\left(j+\frac{\theta}{\alpha}\right)_{t\uparrow1}\frac{(\theta+n+t\alpha)_{(m-tl)\uparrow1}}{(\theta+n)_{m\uparrow1}}\\
&&\mbox{(by Equation \rf{eq_momm})}\\
&&\quad=\sum_{t=0}^{r}S(r,t)\frac{\left(j+\frac{\theta}{\alpha}\right)_{t\uparrow1}}{\left(\frac{\theta+n}{\alpha}\right)_{t\uparrow1}}\E[(M^{\ast}_{l,m})_{t\downarrow 1}]\\
&&\mbox{(by expanding $(j+\theta/\alpha)_{t\uparrow1}/((\theta+n)/\alpha)_{t\uparrow1}$ as an Euler integral)}\\
&&\quad=\sum_{t=0}^{r}S(r,t)\E[(M^{\ast}_{l,m})_{t\downarrow 1}]\frac{\Gamma\left(\frac{\theta+n}{\alpha}\right)}{\Gamma\left(\frac{\theta}{\alpha}+j\right)\Gamma\left(\frac{n}{\alpha}-j\right)}\int_{0}^{1}x^{t+\frac{\theta}{\alpha}+j-1}(1-x)^{\frac{n}{\alpha}-j-1}d\,x\\
&&\quad=\sum_{t=0}^{r}S(r,t)\E[(M^{\ast}_{l,m})_{t\downarrow 1}]\E[(B_{\frac{\theta}{\alpha}+j,\frac{n}{\alpha}-j})^{t}]\\
&&\quad=\E\left[\E\left[\sum_{t=0}^{r}S(r,t)(M^{\ast}_{l,m})_{t\downarrow 1}(B_{\frac{\theta}{\alpha}+j,\frac{n}{\alpha}-j})^{t}\right]\right]\\
&&\mbox{(by Equation \rf{eq_momb})}\\
&&\quad=\E\left[\left(Z_{M^{\ast}_{l,m},B_{\frac{\theta}{\alpha}+j,\frac{n}{\alpha}-j}}\right)^{r}\right]
\eeq
and the proof of the representation \rf{eq_id2} is completed. 

\hfill $\Box$ 
\end{proof}

Now are ready to prove the main result of this section.

\begin{thm}\label{last}
For any $\alpha\in(0,1)$ and $\theta>-\alpha$, the conditional laws of $\frac{\tilde{K}^{(n)}_m}{m^{\alpha}\beta_m}$ and $\frac{\tilde{M}^{(n)}_{m,l}}{m^{\alpha}\beta_m}$ satisfy MDPs that are the same as  $\frac{K_m}{m^{\alpha}\beta_m}$ and $\frac{M_{l,m}}{m^{\alpha}\beta_m}$, respectively, as $m$ tends to infinity.
\end{thm}
 \proof  First observe that the MDPs for  $\frac{K^{\ast}_m}{m^{\alpha}\beta_m}$ and $\frac{M^{\ast}_{m,l}}{m^{\alpha}\beta_m}$ are the same as the corresponding MDPs for $\frac{K_m}{m^{\alpha}\beta_m}$ and $\frac{M_{l,m}}{m^{\alpha}\beta_m}$, respectively. Furthermore, for any $\la \leq 0$ it is not difficult to see that
\begin{align*}
& \lim_{m\ra \infty}\frac{1}{\beta^{1/(1-\alpha)}_m}\ln\mathbb{E}[e^{\la m^{-\alpha} \beta^{\alpha/(1-\alpha)}_m \tilde{K}^{(n)}_m}|K_n=j]\\
&\quad=\lim_{m\ra \infty}\frac{1}{\beta^{1/(1-\alpha)}_m}\ln\mathbb{E}[e^{\la m^{-\alpha} \beta^{\alpha/(1-\alpha)}_m \tilde{M}^{(n)}_{m,l}}|K_n=j]\\
&\quad=0.
\end{align*}
 
Let  $\{Y_i:i\geq 1\}$ be iid Bernoulli with parameter $\eta=B_{\frac{\theta}{\alpha}+j,\frac{n}{\alpha}-j}$. 
 it follows from Theorem~\ref{p-mdp1} that
 \[
 \tilde{K}^{(n)}_{m}\stackrel{\text{d}}{=} \sum_{i=1}^{K_m^{\ast}}Y_i, \ \ 
\tilde{M}_{m.l}^{(n)}  \stackrel{\text{d}}{=}\sum_{i=1}^{M_{l,m}^{\ast}}Y_i.
    \]
 Hence for $\la >0$, 
 \beq
 \mathbb{E}[e^{\la m^{-\alpha}\beta_m^{\alpha/(1-\alpha)}\tilde{K}^{(n)}_{m}}|K_n=j]&\leq&\mathbb{E}[e^{\la m^{-\alpha}\beta_m^{\alpha/(1-\alpha)}K^{\ast}_{m}}]
 \eeq
 and
 \begin{align*}
 &\mathbb{E}[e^{\la m^{-\alpha}\beta_m^{\alpha/(1-\alpha)}\tilde{K}^{(n)}_{m} }|K_n=j]\\ 
 &\quad\mathbb{E}\bigg[\mathbb{E}[\bigg(1-\eta+\eta e^{\la m^{-\alpha}\beta_m^{\alpha/(1-\alpha)}}\bigg)^{K_m^{\ast}}]\bigg]\\ 
 &\quad\geq\mathbb{E}\bigg[e^{\la m^{-\alpha}\beta_m^{\alpha/(1-\alpha)}K^\ast_m}\mathbb{E}[\eta^{K_m^{\ast}}]\bigg]\\
 &\quad\geq  \mathbb{E}\bigg[e^{\la m^{-\alpha}\beta_m^{\alpha/(1-\alpha)}K^\ast_m}\frac{\Gamma(\frac{\theta+n}{\alpha})}{\Gamma(\frac{\theta}{\alpha})}\frac{\Gamma(K_m^{\ast}+\frac{\theta}{\alpha})}{\Gamma(K_m^{\ast}+\frac{\theta+n}{\alpha})}\bigg]\\
 &\quad\geq  \frac{1}{m^{\ga(m,\alpha,\theta,n,j)}}\mathbb{E}[e^{\la m^{-\alpha}\beta_m^{\alpha/(1-\alpha)}K^\ast_m} ]
  \end{align*}
 where $\ga(m,\alpha,\theta,n,j)$ is sequence of positive numbers converging to $\frac{n}{\alpha}-j$ for large $m$.  Thus we have
 \be\label{p-mdp2}
\lim_{m\ra \infty}\frac{1}{\beta^{1/(1-\alpha)}_m}\ln \mathbb{E}[e^{\la m^{-\alpha}\beta_m^{\alpha/(1-\alpha)}\tilde{K}^{(n)}_{m}}|K_n=j]=\la^{1/\alpha}.
 \ee
  Similarly we can show that
  \[
\lim_{m\ra \infty}\frac{1}{\beta^{1/(1-\alpha)}_m}\ln \mathbb{E}[e^{\la m^{-\alpha}\beta_m^{\alpha/(1-\alpha)}\tilde{M}^{(n)}_{m}}|K_n=j]=\bigg(\frac{\alpha(1-\alpha)_{(l-1)\uparrow 1} }{l!}\la\bigg)^{1/\alpha}
 \]    
 which combined with \rf{p-mdp2} led to the theorem.
 
  \hfill $\Box$
  
The MDP results in Theorems \ref{t1}, \ref{t2} and \ref{last} identify a critical scale at  $(\ln m)^{1-\alpha}$. It is not clear whether MDP holds when  $\beta_m$ is at or has a slower growth rate than $(\ln m)^{1-\alpha}$. Our calculations indicate that if such MDPs hold true, then the posterior MDP and the unconditional MDP may be different.

\end{document}